\documentclass[11pt]{article}
\usepackage{mathrsfs}
\usepackage{amssymb}
\usepackage{amsmath}
\usepackage{pict2e}
\usepackage{amsmath}
\usepackage{amsfonts}
\usepackage{graphicx}
\usepackage[small]{caption}
\usepackage{subfigure}
\usepackage{cite}
\usepackage{mathtools}
\usepackage{multirow}

\def\0{\emptyset}

\def\sb{\subseteq}

\def\q{\hfill\rule{1ex}{1ex}}

\begin{document}
\newtheorem{claim}{Claim}[section]
\newtheorem{theorem}{Theorem}[section]
\newtheorem{corollary}[theorem]{Corollary}
\newtheorem{definition}[theorem]{Definition}
\newtheorem{conjecture}[theorem]{Conjecture}
\newtheorem{question}[theorem]{Question}
\newtheorem{lemma}[theorem]{Lemma}
\newtheorem{assump}[theorem]{Assumption}
\newtheorem{proposition}[theorem]{Proposition}
\newtheorem{observation}[theorem]{Observation}
\newtheorem{property}[theorem]{Property}
\newtheorem{Case}[theorem]{Case}
\newenvironment{proof}{\noindent {\bf
Proof.}}{\rule{2mm}{2mm}\par\medskip}
\newcommand{\remark}{\medskip\par\noindent {\bf Remark.~~}}
\newcommand{\pp}{{\it p.}}
\newcommand{\de}{\em}

\title{\bf Fractional matching preclusion of fault Hamiltonian graphs}

\author{Huiqing Liu\thanks{Hubei Key Laboratory of Applied Mathematics, Faculty of Mathematics and Statistics, Hubei University, Wuhan 430062, PR China. E-mail: hqliu@hubu.edu.cn. This author's work was partially supported by NNSFC (Nos. 11571096, 61373019).}\and
Shunzhe Zhang\thanks{Hubei Key Laboratory of Applied Mathematics, Faculty of Mathematics and Statistics, Hubei University, Wuhan 430062, PR China.
        E-mail: shunzhezhang@hubu.edu.cn.}\and
	Xinyuan Zhang\thanks{Artificial intelligence School, Wuchang University of Technology,
        Wuhan 430223, PR China. E-mail: xinyuan0715@163.com.}
        }

\date{}
\maketitle \baselineskip 17.4pt

\date{}
\maketitle \baselineskip 17.4pt

\begin{abstract}

Matching preclusion is a measure of robustness in the event of edge failure in interconnection networks. As a generalization of matching preclusion, the fractional matching preclusion number (FMP number for short) of a graph is the minimum number of edges whose deletion results in a graph that has no fractional perfect matchings, and the fractional strong matching preclusion number (FSMP number for short) of a graph is the minimum number of edges and/or vertices whose deletion leaves a resulting graph with no fractional perfect matchings. A graph $G$ is said to be $f$-fault Hamiltonian if there exists a Hamiltonian cycle in $G-F$ for any set $F$ of vertices and/or edges with $|F|\leq f$. In this paper, we establish the FMP number and FSMP number of $(\delta-2)$-fault Hamiltonian graphs with minimum degree $\delta\geq 3$. As applications, the FMP number and FSMP number of some well-known networks are determined.

\vskip 0.2cm

\end{abstract}

\section{Introduction}
Let $G=(V(G),E(G))$ be a simple, undirected and finite graph. We denote $VE(G)=V(G)\cup E(G)$ and simply write $|V(G)|$ by $|G|$. For $v\in V(G)$, the set of all edges incident with $v$ is denoted by $E_G(v)$ and the \textit{minimum degree} of $G$, denoted by $\delta(G)$, is the minimum size of $|E_G(v)|$. For $S\subseteq V(G)$, $G[S]$ denotes the subgraph induced by $S$ and {\em $G-S$} denotes the subgraph induced by $V(G)\setminus S$. For $F\subseteq E(G)$, {\em $G-F$} denotes the resulting graph by deleting all edges of $F$ from $G$. For $f\in VE(G)$, we simplify $G-\{f\}$ to $G-f$. Denote $F_V=F\cap V(G)$ and $F_E=F\cap E(G)$ for $F\sb VE(G)$. Two graphs are {\em vertex disjoint} if they have no vertex in common. A \textit{$k$-cycle} is a cycle with $k$ vertices. A cycle is called a {\em Hamiltonian cycle} if it contains all vertices of the graph. A graph is said to be {\em Hamiltonian} if it contains a Hamiltonian cycle. A graph $G$ is said to be an {\em $f$-fault Hamiltonian graph} if there exists a Hamiltonian cycle in $G-F$ for any set $F$ of vertices and/or edges with $|F|\leq f$. A graph is said to be {\em even} if it has even number of vertices, otherwise, it is said to be {\em odd}.

A {\em matching }in a graph is a set of edges no two of which are adjacent. With any matching $M$ of a graph $G$, we may associate a $\{0, 1\}$-valued function $f$ that assigns to each edge of $G$ a number in $\{0, 1\}$ such that $\sum_{e\in E_G(v)}f(e) \leq 1$ for each vertex $v\in V(G)$. A matching is {\em perfect} if
$\sum_{e\in E_G(v)}f(e)=1$ for each vertex $v$. A matching is {\em almost-perfect} if there exists exactly one vertex $u$ such that
$\sum_{e\in E_G(u)}f(e)=0$ and $\sum_{e\in E_G(v)}f(e)=1$ for each vertex $v\in V(G-u)$.

A {\em matching preclusion set} (MP set for short) is an edge subset $F$ of $G$ if $G-F$ has neither perfect matchings nor almost-perfect matchings. The {\em MP number} of $G$,
denoted by $mp(G)$, is the minimum size of MP sets of $G$. In 2005, Brigham {\em et al.} \cite{Brigham05} first introduced the matching preclusion problem which offers a way of measuring the robustness of a given graph as a network topology with respect to link failures. That is, in the situation in which each node of a communication network is demanded to have a special partner at any time, one that has a larger matching preclusion number may be considered as more robust in the event of possible link failures. Since then, the matching preclusion problem of various networks was studied, see \cite{Cheng12, HuLiu13, Li2015, Li2016, Wang10}. For any $v\in V(G)$, the set of edges incident with $v$ forms a MP set of $G$ if $|G|$ is even. Thus $mp(G)\leq \delta(G)$.

Another type of failure in a communication network occurs through nodes, which is in fact more offensive, is through node failures. As a more general matching preclusion problem, the strong matching preclusion deals with the corresponding matching problem that has also been analyzed under vertex deletions, see \cite{Aldred07,Guichard08}. Park and Ihm \cite{Park11} considered the following extended form of matching preclusion. A {\em strong matching preclusion set} (SMP set for short) is a set $F$ of edges and/or vertices of $G$ if $G - F$ has neither perfect matchings nor almost-perfect matchings. The {\em SMP number} of $G$, denoted by $smp(G)$, is the minimum size of SMP sets of $G$. According to the definition of $mp(G)$ and $smp(G)$, we have $smp(G) \leq mp(G)$.

By utilizing the definition of matching with the continuous unit interval $[0,1]$ instead of the ``discrete unit interval'' $\{0, 1\}$, we get the following generalization of matching introduced in \cite{FGT}.

A {\em fractional matching} is a function $f$ that assigns to each edge a number in $[0, 1]$ such that $\sum_{e\in E_G(v)}f(e) \leq 1$ for each vertex $v\in V(G)$.
A {\em fractional perfect matching} is a fractional matching $f$ so that $\sum_{e\in E_G(v)}f(e)=1$ for every $v \in V(G)$. Note that a perfect matching is also a fractional perfect matching.

Let $G$ be a graph and $S\subseteq V(G)$, we use $i(G-S)$ and $c(G-S)$ to denote the number of isolated vertices and the number of components of $G-S$, respectively. It is obvious that $i(G-S)\leq c(G-S)$. The following proposition is a necessary and sufficient condition for a graph to have a fractional perfect matching.

\begin{proposition} \cite{FGT}
\label{1.1}
A graph $G$ has a fractional perfect matching if and only if $i(G-S)\leq |S|$ for every set $S\sb V(G)$.
\end{proposition}

As a generalization of matching preclusion, the concept of the fractional matching preclusion number was introduced by Liu {\em et al.} \cite{LiuLiu17} .
A {\em fractional matching preclusion set} (FMP set for short) is an edge subset $F$ of $G$ if $G-F$ has
no fractional perfect matchings. The {\em FMP number} of $G$, denoted by $fmp(G)$, is
the minimum size of FMP sets of $G$. Obviously, $fmp(G)\leq \delta(G)$. By the definition of $fmp(G)$, $mp(G) \leq fmp(G)$ if $|G|$ is even. So we have the following proposition.

\begin{proposition} \cite{LiuLiu17}
\label{1.2}
Let $G$  be a graph of even order. If $mp(G)=\delta(G)$, then $mp(G)=fmp(G)=\delta(G)$.
\end{proposition}

A {\em fractional strong matching preclusion set} (FSMP set for short) is a set $F$ of edges and/or vertices of $G$ if $G-F$ has no fractional perfect matchings. The {\em FSMP number} of $G$, denoted by $fsmp(G)$, is the minimum size of FSMP sets of $G$.  By the definition of $fmp(G)$ and $fsmp(G)$, we have the following proposition.

\begin{proposition} \cite{LiuLiu17}
\label{1.3}
Let $G$  be a graph. Then $fsmp(G)\leq fmp(G)\leq \delta(G)$.
\end{proposition}

In 2017, Liu and Liu \cite{LiuLiu17} considered the FMP number and FSMP number of complete graphs, Petersen graph and twisted cubes. Later, Ma {\em et al.} \cite{Ma18} obtained the FMP number and FSMP number of (burnt) pancake graphs. Ma {\em et al.} \cite{MaMao19} determined the FMP number and FSMP number of arrangement graphs. Recently, Zhang {\em et al.} \cite{Zhang19} established the FMP number and FSMP number of the $n$-dimensional restricted HL-graphs, respectively.

In this paper, we establish the FMP number and FSMP number of $(\delta-2)$-fault Hamiltonian graphs with minimum degree $\delta\geq 3$. As applications, the FMP number and FSMP number of some well-known networks are determined.

The rest of the paper is organized as follows: Section 2 provides some useful lemmas. In Section 3, we investigate the FMP number and FSMP number of $(\delta-2)$-fault Hamiltonian graphs with minimum degree $\delta\geq 3$. In Section 4, we determine the FMP number and FSMP number of some well-known networks. Our conclusions are given in Section 5.

\section{Lemmas}

Mao {\em et al.} \cite{Mao18} gave a sufficient condition to determine the MP number and SMP number of fault Hamiltonian graphs.

\begin{lemma}\cite{Mao18}
\label{H}
Let $G$ be a $(\delta-2)$-fault Hamiltonian graph with minimum degree $\delta$. Then $smp(G)=mp(G)=\delta$.
\end{lemma}

The following lemma shows a necessary condition for the existence of a Hamiltonian cycle.

\begin{lemma} \cite{Bond08}
\label{c(G)}
Let $S$ be a set of vertices of a Hamiltonian graph $G$. Then
\begin{align}
c(G-S)\leq |S|.
\end{align}
Moreover, if equality holds in $(1)$, then each of the $|S|$ components of $G-S$ has a Hamiltonian path.
\end{lemma}

The following is a necessary and sufficient condition for a graph to have a fractional perfect matching.

\begin{lemma} \cite{FGT}
\label{cycle-edge}
A graph $G$ has a fractional perfect matching if and only if
there is a partition $\{V_1, V_2,\ldots, V_s\}$ of the vertex set $V(G)$ such that, for each $i$, the graph $G[V_i]$ is either $K_2$ or Hamiltonian.
\end{lemma}

An \textit{independent set} in a graph is a set of vertices no two of which are adjacent. The cardinality of a maximum independent set in a graph $G$ is called the \textit{independent number} of $G$ and is denoted by $\alpha(G)$. A \textit{covering} of a graph $G$ is a set of vertices which together meet all edges of $G$. The minimum number of vertices in a covering of a graph $G$ is called the \textit{covering number} of $G$ and is denoted by $\beta(G)$. Gallai \cite{Ga59} showed the relationship between the independent number $\alpha(G)$ and the covering number $\beta(G)$ of a graph $G$.

\begin{lemma} \cite{Ga59}
\label{i}
$\alpha(G)+\beta(G)=|G|$ for a graph $G$. Furthermore, $\alpha(G)\geq |G|-|E(G)|$, where the equality holds if and only if $E(G)$ is a matching of $G$.
\end{lemma}

\begin{lemma}
\label{look}
Let $G$ be a graph with $F\sb VE(G)$. If $G-F$ is an independent set, then
\begin{align}
\alpha(G)\geq |G-F|-|F_E|,
\end{align}
where the equality holds if and only if $E(G-F_V)=F_E$ and $F_E$ is a matching of $G-F_V$.
\end{lemma}

\begin{proof} Note that $G-F$ is an independent set, then $E(G-F_V)\subseteq F_E$. By Lemma \ref{i}, $\alpha(G)\geq \alpha(G-F_V)\geq |G-F_V|-|E(G-F_V)|\geq |G-F|-|F_E|$, where the equality holds if and only if $E(G-F_V)=F_E$ and $F_E$ is a matching of $G-F_V$.
\end{proof}

\begin{lemma}
\label{lower}
Let $G$ be a $(\delta-2)$-fault Hamiltonian graph with minimum degree $\delta$ and $F$ a subset of $VE(G)$ with $|F|=\delta-1$. If there exists $S\sb V(G-F)$ such that
$i(G-(F\cup S))\geq |S|+1$, then $G-(F\cup S)$ is an independent set with $\frac{|G|+|F_E|-\delta}{2}+1$ vertices and $\alpha(G)\geq |G-(F\cup S)|-|F_E|$.
\end{lemma}

\begin{proof} Let $F'=F-\beta$ for some $\beta\in F$, and thus $|F'|=\delta-2$. Since $G$ is $(\delta-2)$-fault Hamiltonian, $G-F'$ is Hamiltonian. By Lemma \ref{c(G)}, for any $S'\sb V(G-F')$,
\begin{align}
c(G-(F'\cup S'))\leq |S'|.
\end{align}

Note that $S\sb V(G-F)$ and $G-(F\cup S)=G-(F'\cup (S\cup \{\beta\}))$, where $\beta\in F$. \\If $\beta\in F_V$, then $S\cup \{\beta\}\sb V(G-F')$. This together with $(3)$, we have
\begin{align}
c(G-(F\cup S))=c(G-(F'\cup (S\cup \{\beta\})))\leq |S\cup \{\beta\}|=|S|+1.   \nonumber
\end{align}
If $\beta\in F_E$, then $S\sb V(G-F')$, and thus, by $(3)$,
\begin{align}
c(G-(F\cup S))=c(G-(F'\cup (S\cup \{\beta\})))\leq c(G-(F'\cup S))+1\leq |S|+1.     \nonumber
\end{align}
Then $c(G-(F\cup S))\leq |S|+1$. Note that $i(G-(F\cup S))\geq |S|+1$. Hence,
\begin{align}
|S|+1\leq i(G-(F\cup S))\leq c(G-(F\cup S))\leq |S|+1,   \nonumber
\end{align}
which means $i(G-(F\cup S))=c(G-(F\cup S))=|S|+1$, and thus $G-(F\cup S)$ is an independent set and $|G-(F\cup S)|=|S|+1$. Note that $|G-(F\cup S)|=|G|-|F_V|-|S|$ and $|F_V|+|F_E|=|F|=\delta-1$. Then
$|G-(F\cup S)|=\frac{|G|+|F_E|-\delta}{2}+1$. Therefore, by Lemma \ref{look}, $\alpha(G)\geq |G-(F\cup S)|-|F_E|$.
\end{proof}

%%%%%%%%%%%%%%%%%%%%%%%%%%%%%%%%%%%%%%%%%%%%%%%%%%%%%%%%%%%%%%%%%%%%%%%%%%%%%%%%%%%%%%%%%%%%%%%%%%%%%%%%%%%%%%%%%%%%%%%%%%%%%%%%%%%%%%%%%%%%%%%%%%%%%%

\section{The fractional (strong) matching preclusion of $G$}

The following theorem investigates the FSMP number of regular bipartite graphs.

\begin{theorem}
\label{bi-fsmp}
Let $G$ be a regular bipartite graph. Then $fsmp(G)=1$.
\end{theorem}

\begin{proof} Let $G=G[X,Y]$ be a regular bipartite graph. Then $|X|=|Y|$ and $G$ has a perfect matching (see \cite{Bond08}). By Lemma \ref{cycle-edge}, $G$ has a fractional perfect matching. Thus $fsmp(G)\geq 1$. Let $x\in X$ and $X'=X-x$. Then $i((G-x)-X')=|Y|>|X|-1=|X'|$, and thus, by Proposition \ref{1.1}, $G-x$ has no fractional perfect matchings. Then $fsmp(G)\leq 1$. Hence $fsmp(G)=1$.
\end{proof}

In the following, we always assume that $G$ is $(\delta-2)$-fault Hamiltonian with minimum degree $\delta\geq 3$. By Lemma \ref{H} and Proposition \ref{1.2}, we have the following result directly.

\begin{lemma}
\label{fmp}
If $|G|$ is even, then $fmp(G)=\delta$.
\end{lemma}

The following lemma shows the upper and lower bound of the FMP number and FSMP number of $G$.

\begin{lemma}
\label{range}
$\delta-1\leq fsmp(G)\leq fmp(G)\leq \delta$.
\end{lemma}

\begin{proof} By Proposition \ref{1.3}, $fsmp(G)\leq fmp(G)\leq \delta$ . Let $F\sb VE(G)$ with $|F|\leq \delta-2$. Since $G$ is $(\delta-2)$-fault Hamiltonian, $G-F$ has a Hamiltonian cycle, and thus $G-F$ has a fractional perfect matching by Lemma \ref{cycle-edge}. It follows that $fsmp(G)\geq \delta-1$.
\end{proof}

\begin{lemma}
\label{odd}
Let $F$ be a subset of $VE(G)$ with $|F|=\delta-1$. If $G-F$ has no fractional perfect matchings, then $|G|-|F_V|$ is odd. Furthermore, $|F_E|\leq \delta-2$ when $|G|$ is even.
\end{lemma}

\begin{proof} Suppose that $|G|-|F_V|$ is even. Then $G-F$ is even order. By Lemma \ref{H}, $smp(G)=\delta>|F|$, and hence, $G-F$ has a perfect matching. Then by Lemma \ref{cycle-edge}, $G-F$ has a fractional perfect matching, a contradiction. Furthermore, if $|G|$ is even, then $|F_V|\neq 0$. Note that $|F_V|+|F_E|=|F|=\delta-1$, then $|F_E|\leq \delta-2$.
\end{proof}

Next, we first give the definitions of two different graph classes, then show some sufficient conditions to determine the FMP number and FSMP number of fault Hamiltonian graphs.

A graph $G$ is called $\mathcal{H}$-$\textit{free}$ if $G$ does not contain $H$ as an induced subgraph for any $H\in \mathcal{H}$, and
we call each $H$ a $\textit{forbidden subgraph}$. Let $\mathcal{G}_1(k)=\{G~|~G$ is a $k$-regular odd graph or $\{K_4-e\}$-free even graph in which every edge lies in at least a $3$-cycle and a $4$-cycle$\}$ and $\mathcal{G}_2(4)=\{G~|~G$ is a $4$-regular $\{K_4,K_4-e,K_{2,3}\}$-free odd graph in which every edge lies in at least a $3$-cycle and two $4$-cycles$\}$.

\begin{theorem}
\label{fsmp}
Let $G$ be a $(\delta-2)$-fault Hamiltonian graph with minimum degree $\delta\geq 3$.

$(i)$ If $\alpha(G)\leq \lceil \frac{|G|+1}{2}\rceil -\delta$, then $fsmp(G)=fmp(G)=\delta$;

$(ii)$ If $G\in \mathcal{G}_1(\delta)$ and $\alpha(G)\leq \lceil \frac{|G|+1}{2}\rceil -\delta+1$, then $fsmp(G)=fmp(G)=\delta$;

$(iii)$ If $G\in \mathcal{G}_2(4)$ and $\alpha(G)\leq \frac{|G|+1}{2}-2$, then $fsmp(G)=fmp(G)=4$.

\end{theorem}

\begin{proof} By Lemma \ref{range}, $\delta-1\leq fsmp(G)\leq fmp(G)\leq \delta$. In the following, we will show that, for any $F\sb VE(G)$ with $|F|=\delta-1$, $G-F$ has a fractional perfect matching. Suppose, to the contrary, that $G-F$ has no fractional perfect matchings. By Proposition \ref{1.1}, there exists $S\sb V(G-F)$ such that $i(G-(F\cup S))\geq |S|+1$. By Lemma \ref{lower},
\begin{align}
|G-(F\cup S)|=\frac{|G|+|F_E|-\delta}{2}+1.
\end{align}
and
\begin{align}
\alpha(G)\geq |G-(F\cup S)|-|F_E|=\frac{|G|-|F_E|-\delta}{2}+1.
\end{align}

$(i)$ Note that $|F_E|\leq |F|=\delta-1$, then by the inequality $(5)$, $\alpha(G)\geq \frac{|G|-|F_E|-\delta}{2}+1\geq \frac{|G|+1}{2}-\delta+1$, a contradiction to the assumption.

$(ii)$ By the assumption $\alpha(G)\leq \lceil \frac{|G|+1}{2}\rceil -\delta+1$ and the inequality $(5)$, we have
\begin{align}
\frac{|G|-|F_E|-\delta}{2}+1\leq |G-(F\cup S)|-|F_E|\leq \alpha(G)\leq \lceil \frac{|G|+1}{2}\rceil -\delta+1,
\end{align}
which implies $|F_E|=\delta-1$ when $|G|$ is odd and $\delta-2\leq |F_E|\leq \delta-1$ when $|G|$ is even, and thus $|F_E|=\delta-2$ as Lemma \ref{odd}. Combining this with the inequality $(6)$, we have
\begin{align}
\alpha(G)=|G-(F\cup S)|-|F_E|=\lceil\frac{|G|+1}{2}\rceil -\delta+1.
\end{align}
Consider a partition $\{V(G)-(F_V\cup S), F_V\cup S\}$ of $V(G)$. Since $G$ is $\delta$-regular, we can deduce $|E(G[F_V\cup S])|=|E(G)|-(\delta|G-(F\cup S)|-|F_E|)$.

If $|G|$ is odd, then $|F_E|=\delta-1\geq 2$ as $\delta\geq 3$. By $(4)$, $|G-(F\cup S)|=\frac{|G|+1}{2}$, and thus $|F_V\cup S|=\frac{|G|-1}{2}$. Then $|E(G[F_V\cup S])|=\frac{\delta|G|}{2}-(\frac{\delta(|G|+1)}{2}-(\delta-1))=\frac{\delta}{2}-1$. This together with $(7)$, we have $|G[F_V\cup S]|-|E(G[F_V\cup S])|=\frac{|G|-1}{2}-(\frac{\delta}{2}-1)=\frac{|G|+1}{2}-\frac{\delta}{2}>\lceil\frac{|G|+1}{2}\rceil -\delta+1=\alpha(G)\geq \alpha(G[F_V\cup S])$, a contradiction to Lemma \ref{i}.

If $|G|$ is even, then $|F_E|=\delta-2\geq 1$ as $\delta\geq 3$. By $(4)$, $|G-(F\cup S)|=\frac{|G|}{2}$, and thus $|F_V\cup S|=\frac{|G|}{2}$. Then $|E(G[F_V\cup S])|=\frac{\delta|G|}{2}-(\frac{\delta|G|}{2}-(\delta-2))=\delta-2$. Note that $G-(F\cup S)$ is an independent set as Lemma \ref{lower}. By the equality $(7)$ and Observation \ref{look}, $E(G-(F_V\cup S))=F_E$ and $F_E$ is a matching of $G-(F_V\cup S)$. This implies every edge of $G-(F_V\cup S)$ lies in a $4$-cycle which must contain an edge of $G[F_V\cup S]$.
Suppose that $uvxy$ is a $4$-cycle with $uv\in E(G-(F_V\cup S))$ and $xy\in E(G[F_V\cup S])$. By the equality $(7)$ and Lemma \ref{i}, $\frac{|G|+2}{2}-\delta+1=\alpha(G)\geq \alpha(G[F_V\cup S])\geq |G[F_V\cup S]|-|E(G-(F_V\cup S))|=\frac{|G|}{2}-(\delta-2)=\frac{|G|+2}{2}-\delta+1$. Then, by Lemma \ref{i}, $E(G[F_V\cup S])$ is a matching of $G[F_V\cup S]$. Since $G$ is $\{K_4-e\}$-free, then there is no $3$-cycles containing $vx$ in $G$, a contradiction to $G\in \mathcal{G}_1(\delta)$.

$(iii)$ By the assumption $\alpha(G)\leq \frac{|G|+1}{2}-2$ and the inequality $(5)$, we have
\begin{align}
\frac{|G|-|F_E|}{2}-1\leq \alpha(G)\leq \frac{|G|-1}{2}-1,
\end{align}
which implies $|F_E|\geq 1$. Note that $|F_E|\leq |F|=3$ and $|G|$ is odd, then by Lemma \ref{odd}, $|F_V|\neq 1$, and thus $|F_E|\neq 2$. This implies $|F_E|=1$ or $|F_E|=3$. Since $G$ is $4$-regular, we can deduce $|E(G[F_V\cup S])|=|E(G)|-(4|G-(F\cup S)|-|E(G-(F_V\cup S))|)$.

%Note that $G-(F\cup S)$ is an independent set of $G-(F_V\cup S)$ as Lemma \ref{lower}, then $E(G-(F_V\cup S))\subseteq F_E$.

If $|F_E|=1$, then $|E(G-(F_V\cup S))|\leq |F_E|=1$. Assume that $|E(G-(F_V\cup S))|=0$. Then by $(4)$, $\alpha(G)\geq |G-(F_V\cup S)|=\frac{|G|+|F_E|-\delta}{2}+1= \frac{|G|-1}{2}$, a contradiction to $(8)$. Thus $|E(G-(F_V\cup S))|=|F_E|=1$. By $(4)$, we have $|G-(F\cup S)|=\frac{|G|-1}{2}$. Then $|E(G[F_V\cup S])|=\frac{4|G|}{2}-(\frac{4(|G|-1)}{2}-1)=3$. Note that $|E(G-(F_V\cup S))|=1$, then every edge of $E(G-(F_V\cup S))$ lies in a $4$-cycle which must contain an edge of $G[F_V\cup S]$. Let $u_1v_1x_1y_1$ and $u_1v_1w_1z_1$ be two $4$-cycles with $u_1v_1\in E(G-(F_V\cup S))$ and $x_1y_1, w_1z_1\in E(G[F_V\cup S])$, where $x_1,y_1, w_1,z_1$ are four distinct vertices as $G$ is $\{K_4,K_4-e,K_{2,3}\}$-free. Note that $v_1x_1$ lies in a $3$-cycle. Let $v_1x_1v'$ be a $3$-cycle with $x_1v'\in E(G[F_V\cup S])$ and $v'\neq y_1$ as $G$ is $\{K_4-e\}$-free. Recall that $|E(G[F_V\cup S])|=3$, there is no $3$-cycles containing $u_1y_1$ in $G$, a contradiction to $G\in \mathcal{G}_2(4)$.

If $|F_E|=3$, then $|E(G-(F_V\cup S))|\leq |F_E|=3$. By $(4)$, we have $|G-(F\cup S)|=\frac{|G|+1}{2}$, and thus $|F_V\cup S|=|G|-|G-(F\cup S)|=\frac{|G|-1}{2}$. Then $|E(G[F_V\cup S])|\leq \frac{4|G|}{2}-(\frac{4(|G|+1)}{2}-3)=1$. Assume that $|E(G[F_V\cup S])|=0$. Then $\alpha(G)\geq |F_V\cup S|=\frac{|G|-1}{2}$, a contradiction to $(8)$. Thus $|E(G[F_V\cup S])|=1$. This implies every edge of $G[F_V\cup S]$ lies in a $4$-cycle which must contain an edge of $E(G-(F_V\cup S))$. Let $u_2v_2x_2y_2$ and $u_2v_2w_2z_2$ be two $4$-cycles with $u_2v_2\in E(G[F_V\cup S])$ and $x_2y_2, w_2z_2\in E(G-(F_V\cup S))$, where $x_2,y_2, w_2,z_2$ are four distinct vertices as $G$ is $\{K_4,K_4-e,K_{2,3}\}$-free. Note that $v_2x_2$ lies in a $3$-cycle. Let $v_2x_2v''$ be a $3$-cycle with $x_2v''\in E(G-(F_V\cup S))$ and $v''\neq y_2$ as $G$ is $\{K_4-e\}$-free. Since $|E(G-(F_V\cup S))|\leq 3$, there is no $3$-cycles containing $u_2y_2$ in $G$, a contradiction to $G\in \mathcal{G}_2(4)$.
\end{proof}

%%%%%%%%%%%%%%%%%%%%%%%%%%%%%%%%%%%%%%%%%%%%%%%%%%%%%%%%%%%%%%%%%%%%%%%%%%%%%%%%%%%%%%%%%%%%%%%%%%%%%%%%%%%%%%%%%%%%%%%%%%%%%%%%%%%%%%%%%%%%%%%%%%%%%%

\section{Applications to some networks}

In the following, we will determine the FMP number and FSMP number of some well-known networks by the conclusions in Section $3$.

\subsection{Restricted HL-graphs}

The restricted HL-graph is defined using a special graph construction operator.
Given two graphs $G_0$ and $G_1$, consider a set $\Phi(G_0, G_1)$, made of all bijections from $V(G_0$) to $V(G_1)$. Then, given a bijection $\phi\in \Phi(G_0, G_1)$, we denote by $G_0\oplus_{\phi} G_1$ a graph whose vertex set is $V(G_0)\cup V(G_1)$ and edge set is $E(G_0)\cup E(G_1)\cup \{(v, \phi(v)): v\in V(G_0)\}$. Based on the graph constructor, Vaidya {\em et al.} \cite{VaRa93} gave a recursive definition of a class of graphs as follows.

\begin{definition}\cite{VaRa93}
\label{4.1}
Let $RHL_0=\{K_1\}$, $RHL_1=\{K_2\}$, $RHL_2=\{C_4\}$, $RHL_3=\{G(8,4)\}$ and $RHL_n=\{G_0\oplus_{\phi} G_1: G_0, G_1\in  RHL_{n-1}, \phi\in \Phi(G_0, G_1)\}$ for $n\geq 4$. A graph that belongs to $RHL_n$, denoted by $G^n$, is called an $n$-dimensional restricted HL-graph.
\end{definition}

\begin{figure}[!htb]
\centering
{\includegraphics[height=0.37\textwidth]{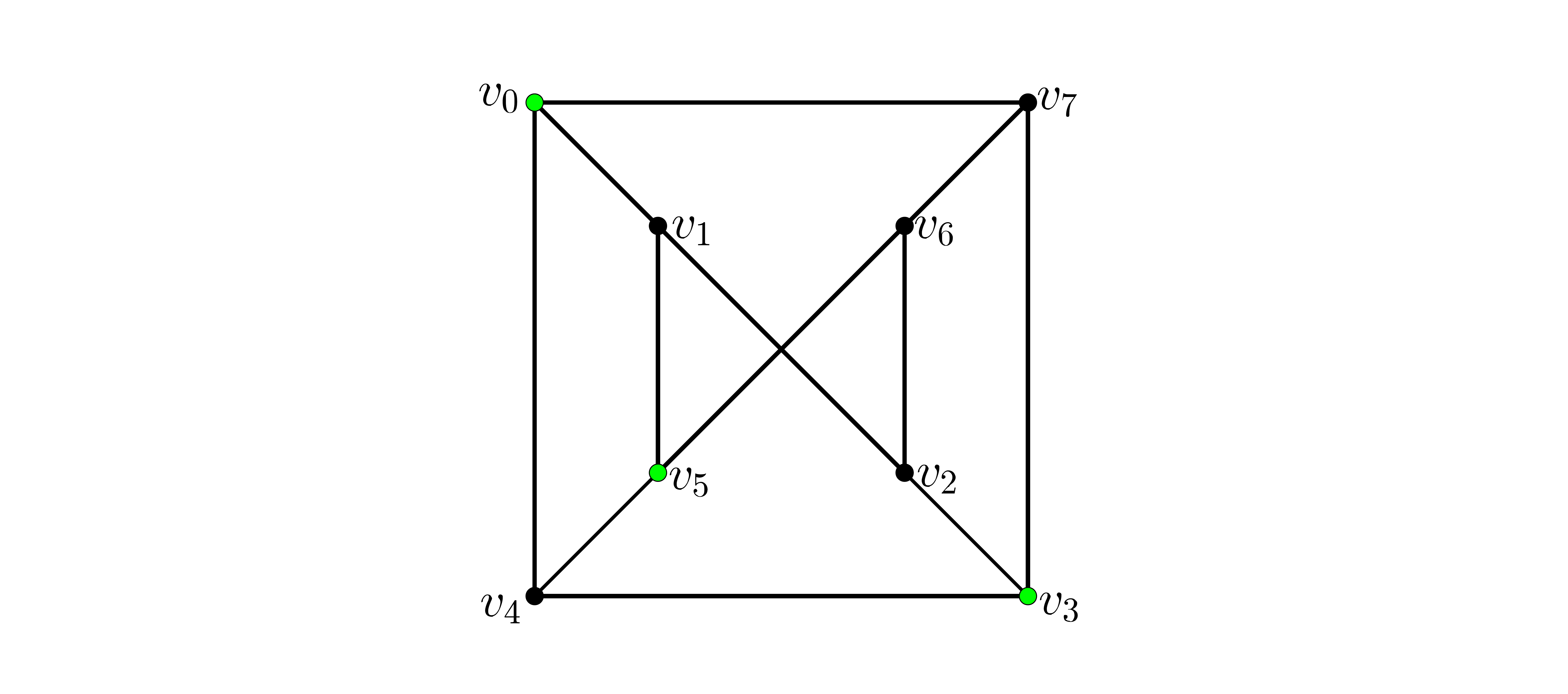}}

\vskip -.3cm

Fig.~1~~$G(8,4)$
\end{figure}

The graph $G(8,4)$ is shown in Fig.~1. $G^n$ is $n$-regular with $2^n$ vertices. Many of the non-bipartite hypercube-like interconnection networks such as crossed cube \cite{Efe}, M\"{o}bius cube \cite{CullLarson}, twisted cube \cite{Hilbers}, multiply twisted cube \cite{Efe91}, generalized twisted cube \cite{Chedid}, locally twisted cube \cite{Yang05}, the twisted hypercubes \cite{Zhu17} etc. proposed in the literature are restricted HL-graphs.

\begin{lemma}
\label{4.3}
$\alpha(G^n)\leq 3\times 2^{n-3}$ for $n\geq 3$.
\end{lemma}

\begin{proof} Clearly, $\alpha(G^3)=3$ (see Fig.~1, the green dots represent the vertices of a maximum independent set of $G^3$). By induction, suppose that $\alpha(G^{n-1})\leq 3\times 2^{n-4}$ for $n\geq 4$. Note that $G^n$ can be decomposed into two vertex disjoint subgraphs each of which is isomorphic to $G_{n-1}$, and thus every independent set of $G^n$ contains at most $2\times (3\times 2^{n-4})=3\times 2^{n-3}$ vertices. Then $\alpha(G^n)\leq 3\times 2^{n-3}$ for $n\geq 3$.
\end{proof}

Park {\em et al.} \cite{Park05} considered Hamiltonian properties in faulty restricted HL-graphs.

\begin{lemma}  \cite{Park05}
\label{4.2}
$G^n$ is $(n-2)$-fault Hamiltonian for $n\geq 3$.
\end{lemma}

Now we can determine the FMP number and FSMP number of $G^n$, which was also obtained in \cite{Zhang19}.

\begin{theorem}
\label{4.4}
$fsmp(G^n)=fmp(G^n)=n$ for $n\geq 5$.
\end{theorem}

\begin{proof} Note that $G^n$ is $n$-regular with $2^n$ vertices, and thus $|G^n|$ is even. By Lemmas \ref{4.3}, \ref{4.2} and Theorem \ref{fsmp}$(i)$, it suffices to show that
\begin{align}
3\times 2^{n-3}\leq \frac{2^n+2}{2}-n,   \nonumber
\end{align}
which implies $2^{n-3}-n+1\geq 0$. It is obvious that the inequality $2^{n-3}-n+1\geq 0$ holds if $n\geq 5$. Therefore, we complete the proof of Theorem \ref{4.4}.
\end{proof}

%%%%%%%%%%%%%%%%%%%%%%%%%%%%%%%%%%%%%%%%%%%%%%%%%%%%%%%%%%%%%%%%%%%%%%%%%%%%%%%%%%%%%%%%%%%%%%%%%%%%%%%%%%%%%%%%%%%%%%%%%%%%%%%%%%%%%%%%%%%%%%%%%%%%%%%%%%%%%%%%%%%%%%%%%%%%%%%%%%%%

\subsection{Torus networks}

Torus networks have been proved to be a viable choice for the interconnection networks, such as ease of implementation, low latency, and high-bandwidth inter-processor communication.

\begin{definition}\cite{VaRa93}
\label{4.5}
Given $k_1,\ldots,k_n$ with $k_i\geq 3$, the $n$-dimensional torus, denoted by $T_{k_1,\ldots ,k_n}$, has the set $\{v_{x_1,x_2,\ldots, x_n}~|~0\leq x_i\leq k_i-1,~1\leq i\leq n\}$ as its vertex set. Two vertices $x_{a_1,a_2,\ldots,a_n}$ and $y_{b_1,b_2,\ldots,b_n}$ are adjacent if there exists an integer $i$ with $1\leq i\leq n$ such that $a_i-b_i=\pm1\pmod {k_i}$, and $a_j=b_j$ for $1\leq j\neq i\leq n$.
\end{definition}

\begin{figure}[!htb]
\centering
{\includegraphics[height=0.45\textwidth]{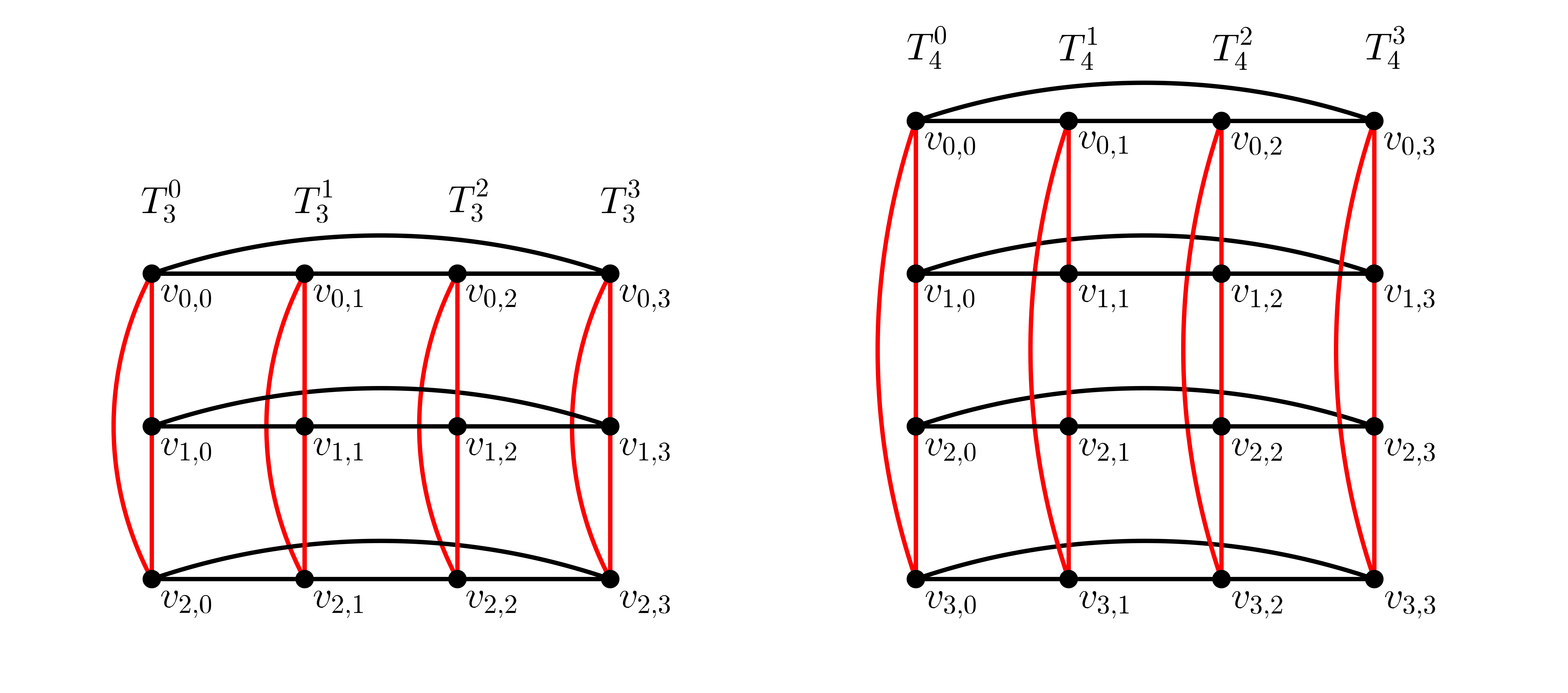}}

\vskip -.5cm

~~~~$(a)$~$T_{3,4}$~~~~~~~~~~~~~~~~~~~~~~~~~~~~~~~~~~~~~~~~~~~~~~~$(b)$~$T_{4,4}$~~~~~~~~

\vskip .4cm

Fig.~2~~Two $2$-dimensional torus networks

\end{figure}

$T_{k_1,\ldots ,k_n}$ is $(2n)$-regular with $k_1\cdots k_n$ vertices and vertex transitive. $T_{k_1,\ldots ,k_n}$ is bipartite if and only if all $k_1,\ldots,k_n$ are even. Fig.~2 shows the $2$-dimensional torus $T_{3,4}$ which is non-bipartite, and $T_{4,4}$ which is bipartite. In particular, the $n$-dimensional torus $T_{k_1,\ldots ,k_n}$ is said to be a $k$-ary $n$-cube $Q^k_n$ if $k_i=k$ for each $1\leq i\leq n$. Let $T^i_{k_1,\ldots ,k_{n-1}}$ be the subgraph of $T_{k_1,\ldots ,k_n}$ induced by the vertices with $i$ in the $n$-th position where $0\leq i\leq k_n-1$. Then $T_{k_1,\ldots ,k_n}$ can be decomposed into $k_n$ vertex disjoint such subgraphs, each of which is isomorphic to $T_{k_1,\ldots ,k_{n-1}}$ (see Fig.~2, each red cycle in $T_{3,4}$ represents $T^i_{3}$ for $0\leq i\leq 3$ and each red cycle in $T_{4,4}$ represents $T^j_{4}$ for $0\leq j\leq 3$).

\begin{lemma}
\label{4.9}
$\alpha(T_{k_1,\ldots,k_n})\leq \frac{(k_1-1)k_2\cdots k_n}{2}$ for $n\geq 1$.
\end{lemma}

\begin{proof} Clearly, $\alpha(T_{k_1})=\frac{k_1-1}{2}$. By induction, suppose that $\alpha(T_{k_1,\ldots ,k_{n-1}})\leq \frac{(k_1-1)k_2\cdots k_{n-1}}{2}$ for $n\geq 2$. Note that $T_{k_1,\ldots,k_n}$ can be decomposed into $k_n$ vertex disjoint subgraphs, each of which is isomorphic to $T_{k_1,\ldots ,k_{n-1}}$, and thus every independent set of $T_{k_1,\ldots,k_n}$ contains at most $k_n\times \frac{(k_1-1)k_2\cdots k_{n-1}}{2}=\frac{(k_1-1)k_2\cdots k_n}{2}$ vertices. Then $\alpha(T_{k_1,\ldots,k_n})\leq \frac{(k_1-1)k_2\cdots k_n}{2}$ for $n\geq 1$.
\end{proof}

Kim and Park \cite{Kim01} considered Hamiltonian properties in $n$-dimensional torus networks with faults.

\begin{lemma}\cite{Kim01}
\label{4.6}
$T_{k_1,\ldots ,k_n}$ is $(2n-2)$-fault Hamiltonian for $n\geq 2$.
\end{lemma}

By Theorem \ref{bi-fsmp}, Lemmas \ref{fmp} and \ref{4.6}, we can determine the FMP number and FSMP number of bipartite $n$-dimensional torus networks easily.

\begin{theorem}
\label{4.7}
Let $T_{k_1,\ldots ,k_n}$ be bipartite with $n\geq 2$. Then
\begin{eqnarray}
fmp(T_{k_1,\ldots ,k_n})=2n~~\mbox{and}~~fsmp(T_{k_1,\ldots ,k_n})=1.  \nonumber
\end{eqnarray}
\end{theorem}

In the following, we always assume $T_{k_1,\ldots ,k_n}$ is non-bipartite, and thus $k_i$ is odd for some $1\leq i\leq n$. By symmetry, we can choose $k_1,\ldots, k_n$ such that

$(T1)$~~ $k_i$ is odd for $1\leq i\leq t$ with $t$ as large as possible.

$(T2)$~~ $k_1\leq \ldots \leq k_t$, subject to $(T1)$.

Therefore, $k_1$ is odd. The following lemma gives the FMP number and FSMP number of $T_{k_1,4}$.

\begin{figure}[!htb]
\centering
{\includegraphics[height=0.45\textwidth]{3}}

~~~$(a)$~$T_{3,4}-F$~~~~~~~~~~~~~~~~~~~~~~~~~~~~~~~~~$(b)$~$T_{k_1,4}-F$~with~$k_1\geq 5$~~~~

\vskip  .5cm

Fig.~3~~$T_{k_1,4}$ with $k_1\geq 3$ and $k_1$ is odd.

\end{figure}

\begin{lemma}
\label{4.8}
$fsmp(T_{k_1,4})=3$ and $fmp(T_{k_1,4})=4$.
\end{lemma}

\begin{proof}
Let $F=\{v_{0,0},v_{1,0}v_{2,0},v_{1,2}v_{2,2}\}$ and $S_1=\{v_{0,2},v_{1,1},v_{1,3},v_{2,1},v_{2,3}\}$. Then $i(T_{3,4}-F-S_1)=6>5=|S_1|$ (see Fig.~3$(a)$, the red dots represent the vertices of $S_1$, the black dots represent the isolated vertices of $T_{3,4}-F-S_1$), and thus, by Proposition \ref{1.1}, $T_{3,4}-F$ has no fractional perfect matchings. It follows that $F$ is an $FSMP$ set of $T_{3,4}$, then $fsmp(T_{3,4})\leq 3$. Note that $fsmp(T_{3,4})\geq 3$ by Lemma \ref{range}. Hence $fsmp(T_{3,4})=3$.

Let $S_2=\{v_{0,2},v_{1,1},v_{1,3},v_{2,1},v_{2,3}\}\cup\{v_{3,0},v_{3,2}\}\cup\cdots\cup\{v_{k_1-1,1},v_{k_1-1,3}\}$. Then $i(T_{k_1,4}-F-S_2)=2k_1>2k_1-1=|S_2|$ (see Fig.~3$(b)$, the red dots represent the vertices of $S_2$, the black dots represent the isolated vertices of $T_{k_1,4}-F-S_2$), and thus, by Proposition \ref{1.1}, $T_{k_1,4}-F$ has no fractional perfect matchings. It follows that $F$ is an $FSMP$ set of $T_{k_1,4}$, then $fsmp(T_{k_1,4})\leq 3$. Note that $fsmp(T_{k_1,4})\geq 3$ by Lemma \ref{range}. Hence $fsmp(T_{k_1,4})=3$.

Note that $T_{k_1,4}$ is 4-regular graph with $4k_1$ vertices. By Lemmas \ref{fmp} and \ref{4.6}, $fmp(T_{k_1,4})=4$. Therefore, we complete the proof of Lemma \ref{4.8}.
\end{proof}

\begin{lemma}
\label{4.10}
$fsmp(T_{k_1,k_2})=fmp(T_{k_1,k_2})=4$ for $k_2\neq 4$.
\end{lemma}

\begin{proof} If $k_2$ is even, then $|T_{k_1,k_2}|=k_1k_2$ is even and $k_2\geq 6$ as $k_2\neq 4$. By Lemma \ref{4.9},
$\alpha(T_{k_1,k_2})\leq \frac{(k_1-1)k_2}{2}\leq \frac{k_1k_2+2}{2}-4$. Then $fsmp(T_{k_1,k_2})=fmp(T_{k_1,k_2})=4$ as Theorem \ref{fsmp}$(i)$.

If $k_2$ is odd, then $|T_{k_1,k_2}|=k_1k_2$ is odd. By Lemma \ref{4.9}, $\alpha(T_{3,3})\leq 3=\frac{9+1}{2}-4+2$. Note that $T_{3,3}\in \mathcal{G}_2(4)$, then $fsmp(T_{3,3})=fmp(T_{3,3})=4$ as Theorem \ref{fsmp}$(iii)$. Now, suppose that $T_{k_1,k_2}\neq T_{3,3}$, and thus $k_2\geq 5$. By Lemma \ref{4.9}, $\alpha(T_{k_1,k_2})\leq \frac{(k_1-1)k_2}{2}\leq \frac{k_1k_2+1}{2}-4+1$. Note that $T_{k_1,k_2}\in \mathcal{G}_1(4)$, then $fsmp(T_{k_1,k_2})=fmp(T_{k_1,k_2})=4$ as Theorem \ref{fsmp}$(ii)$.
\end{proof}

\begin{lemma}
\label{4.11}
$fsmp(T_{k_1,\ldots,k_n})=fmp(T_{k_1,\ldots,k_n})=2n$ for $n\geq 3$.
\end{lemma}

\begin{proof} Note that $T_{k_1,\ldots,k_n}$ is $(2n)$-regular with $k_1\cdots k_n$ vertices. By Lemmas \ref{4.9}, \ref{4.6} and Theorem \ref{fsmp}$(i)$, it suffices to show that
\begin{align}
\frac{(k_1-1)k_2\cdots k_n}{2}\leq \frac{k_1\cdots k_n+1}{2}-2n,   \nonumber
\end{align}
which implies $k_2\cdots k_n-4n+1\geq 0$. It is obvious that the inequality $k_2\cdots k_n-4n+1\geq 0$ holds if $n\geq 4$. Particularly, if $n=3$, then the inequality $k_2k_3-4\times 3+1\geq 0$ holds if $(k_2,k_3)\neq (3,3)$. By the choice of $k_i$, we have $k_1\neq 3$. By Lemma \ref{4.9}, $\alpha(T_{3,3,3})\leq 9=\frac{27+1}{2}-6+1$. Note that $T_{3,3,3}\in \mathcal{G}_1(6)$, then $fsmp(T_{3,3,3})=fmp(T_{3,3,3})=6$ as Theorem \ref{fsmp}$(ii)$. Therefore, we complete the proof of Lemma \ref{4.11}.
\end{proof}

Combining with Lemmas~\ref{4.8},~\ref{4.10} and \ref{4.11}, we obtain the FMP number and FSMP number of non-bipartite $n$-dimensional torus networks.

\begin{theorem}
\label{4.12}
Let $T_{k_1,\ldots ,k_n}$ be non-bipartite with $n\geq 2$. Then $fmp(T_{k_1,\ldots ,k_n})=2n$ and
\begin{eqnarray}
  fsmp(T_{k_1,\ldots,k_n})=\left\{
\begin{array}{ll}
   3,           &\mbox{if}~(n,k_2)=(2,4);   \\          \nonumber
   2n,          &otherwise.                                                   \nonumber
   \end{array}\right.
\end{eqnarray}
\end{theorem}

%%%%%%%%%%%%%%%%%%%%%%%%%%%%%%%%%%%%%%%%%%%%%%%%%%%%%%%%%%%%%%%%%%%%%%%%%%%%%%%%%%%%%%%%%%%%%%%%%%%%%%%%%%%%%%%%%%%%%%%%%%%%%%%%%%%%%%%%%%%%%%%%%%%%%%%%%%%%%%%%%%%%%%%%%%%%%%%%%%%%

\subsection{Recursive circulant graphs}

The recursive circulant graph has many nice properties, such as, vertex transitive, strongly hierarchical, higher connectivity which increases the fault tolerance, smaller diameter which reduces the transmission delay, etc. Park and Chwa \cite{Park94} first introduced the concept of the recursive circulant graph $G(cd^n,d)$ with $n\geq 1$ and $1\leq c<d$ as follows.

\begin{definition} \cite{Park94}
\label{4.13}
The recursive circulant graph $G(cd^n,d)$ with $n\geq 1$ and $1\leq c<d$, has the vertex set $V=\{0, 1, \ldots, {cd^n-1}\}$, and the edge set $E=\{~(u,v)~|~u-v\equiv cd^i \pmod {cd^n}~\mbox{and}~0\leq i\leq \lceil \log_{d}c \rceil+n-1\}$.
\end{definition}

\begin{figure}[!htb]
\centering
{\includegraphics[height=0.42\textwidth]{4}}

\vskip -.5cm

~~~~~~~~~~~~~~~~~~~~~~~~~~$(a)~G(3^2,3)$~~~~~~~~~~~~~~~~~~~~~~~~$(b)~\mbox{The~relabelled}~G(3^2,3)$~~~~~~~~~~~~~~~~~~
\vskip .5cm

Fig.~4~~Two recursive circulant graphs $G(3^2,3)$

\end{figure}

\subsubsection{The recursive circulant graph $G(d^n, d)$}

By the definition of \ref{4.13}, when $c=1$, $G(d^n, d)$ with $n\geq 1$, $d\geq 3$ is $(2n)$-regular with $d^n$ vertices and vertex transitive. The recursive circulant graph $G(3^2, 3)$ is shown in Fig.~4(a). Let $V^{n-1}_{d,i}=\{~v\in V(G(d^n, d))~|~v \equiv i \pmod d\}$ and $G^{n-1}_{d,i}$ be the subgraph of $G(d^n, d)$ induced by the vertices of $V^{n-1}_{d,i}$. Park and Chwa \cite{Park00} proved that $G(d^n, d)$ can be decomposed into $d$ vertex disjoint such subgraphs, each of which is isomorphic to $G(d^{n-1}, d)$ (see Fig.~4$(a)$, each red odd cycle in $G(3^2, 3)$ represents $G^{1}_{3,i}$ for $0\leq i<3$).

In order to be convenient to study the fractional (strong) matching preclusion of $G(d^2,d)$, we relabel vertices of $G(d^2,d)$ such that the vertex $i\times d+j$ corresponds to $v_{i,j}$ for $0\leq i,j<d$. Therefore, the vertex set of $G(d^2,d)$ is represented as
\begin{align}
V(G(d^2,d))=\{v_{i,j}: 0\leq i,j<d\}.   \nonumber
\end{align}
The edge set of $G(d^2,d)$ is classified into two sets:
\begin{align}
E_1=&\{v_{i,k}v_{j,k}:~~~~~~\mathclap{j\equiv i+1}~~~~\pmod d ~\mbox{and}~ 0\leq i,j,k<d\};        \nonumber  \\
E_2=&\{v_{i,j}v_{i,k}:  k=j+1 ~\mbox{and}~ 0\leq i,j,k<d\}~\cup                      \nonumber  \\
    &\{v_{i,0}v_{j,d-1}:~~~~~~\mathclap{j\equiv i-1}~~~~\pmod d ~\mbox{and}~ 0\leq i,j<d\}.        \nonumber
\end{align}
In Fig.~4$(b)$, The vertex $3i+j$ corresponds to $v_{i,j}$ for $0\leq i,j<3$.

\begin{lemma}
\label{4.17}
$\alpha(G(d^2,d))=\frac{d^2-d}{2}$.
\end{lemma}

\begin{proof} See Appendix.
\end{proof}

\begin{lemma}
\label{4.18}
$\alpha(G(d^n,d))\leq \frac{d^n-d^{n-1}}{2}$ for $n\geq 2$.
\end{lemma}

\begin{proof} The conclusion holds for $n=2$ as Lemma \ref{4.17}. By induction on $n$, suppose that $\alpha(G(d^{n-1},d))\leq \frac{d^{n-1}-d^{n-2}}{2}$ for $n\geq 3$.  Note that $G(d^n,d)$ can be decomposed into $d$ vertex disjoint subgraphs, each of which is isomorphic to $G(d^{n-1},d)$, and thus every independent set of $G(d^n,d)$ contains at most $d\times \frac{d^{n-1}-d^{n-2}}{2}=\frac{d^n-d^{n-1}}{2}$ vertices. Then $\alpha(G(d^n,d))\leq \frac{d^n-d^{n-1}}{2}$ for $n\geq 2$ and $d\geq 3$.
\end{proof}

Tsai {\em et al.} \cite{Tsai02} researched Hamiltonian properties of faulty recursive circulant graphs.

\begin{lemma}\cite{Tsai02}
\label{4.14}
$G(d^n,d)$ is $(2n-2)$-fault Hamiltonian for $n\geq 2$.
\end{lemma}

The following two lemmas determine the FMP number and FSMP number of $G(3^2,3)$ and $G(4^2,4)$, respectively.

\begin{figure}[!htb]
\centering
{\includegraphics[height=0.45\textwidth]{5}}

\vskip -.5cm

~~~~~~~~~~~~~~~~~~~~~$(a)$~~~~~~~~~~~~~~~~~~~~~~~~~~~~~~~~~~~~~~~~~~~~~~~~~~~~$(b)$~~~~~~~~~~~~~~~~~~

\vskip  .5cm

Fig.~5~~$G(3^2,3)-F_i$ with $i=1,2$.

\end{figure}

\begin{lemma}
\label{4.15}
$fsmp(G(3^2,3))=fmp(G(3^2,3))=3$.
\end{lemma}

\begin{proof} First we show that $fsmp(G(3^2,3))=3$. Let $F_1=\{v_{0,0},v_{1,0},v_{0,2}v_{2,2}\}$ and $S_1=\{v_{0,1},v_{1,2},v_{2,1}\}$. Then $i(G(3^2,3)-F_1-S_1)=4>3=|S_1|$ (see Fig.~5$(a)$, the red dots represent the vertices of $S_1$, the black dots represent the isolated vertices of $G(3^2,3)-F_1-S_1$), and thus, by Proposition \ref{1.1}, $G(3^2,3)-F_1$ has no fractional perfect matchings. It follows that $fsmp(G(3^2,3))\leq 3$. Note that $fsmp(G(3^2,3))\geq 3$ by Lemma \ref{range}. Hence $fsmp(G(3^2,3))=3$.

Let $F_2=\{v_{0,0}v_{2,0},v_{0,0}v_{2,2},v_{0,2}v_{2,2}\}$ and $S_2=\{v_{0,1},v_{1,0},v_{1,2},v_{2,1}\}$. Then $i(G(3^2,3)-F_2-S_2)=5>4=|S_2|$ (see Fig.~5$(b)$, the red dots represent the vertices of $S_2$, the black dots represent the isolated vertices of $G(3^2,3)-F_2-S_2$), and thus, by Proposition \ref{1.1}, $G(3^2,3)-F_2$ has no fractional perfect matchings. It follows that $fmp(G(3^2,3))\leq 3$. Note that $fmp(G(3^2,3))\geq 3$ by Lemma \ref{range}. Hence $fmp(G(3^2,3))=3$.
\end{proof}

\begin{figure}[!htb]
\centering
{\includegraphics[height=0.45\textwidth]{6}}

Fig.~6~~$G(4^2,4)-\{v_{0,0},v_{1,0}v_{0,3},v_{3,0}v_{2,3}\}$

\end{figure}

\begin{lemma}
\label{4.16}
$fsmp(G(4^2,4))=3$ and $fmp(G(4^2,4))=4$.
\end{lemma}

\begin{proof} By Lemmas \ref{fmp} and \ref{4.14}, $fmp(G(4^2,4))=4$. Now, we show that $fsmp(G(4^2,4))=3$. Let $F_3=\{v_{0,0},v_{1,0}v_{0,3},v_{3,0}v_{2,3}\}$ and $S_3=\{v_{0,2},v_{1,1},v_{1,3},v_{2,0},v_{2,2},v_{3,1},v_{3,3}\}$. Then $i(G(4^2,4)-F_3-S_3)=8>7=|S_3|$ (see Fig.~6, the red dots represent the vertices of $S_3$, the black dots represent the isolated vertices of $G(4^2,4)-F_3-S_3$), and thus, by Proposition \ref{1.1}, $G(4^2,4)-F_3$ has no fractional perfect matchings. It follows that $fsmp(G(4^2,4))\leq 3$. Note that $fsmp(G(4^2,4))\geq 3$ by Lemma \ref{range}. Hence $fsmp(G(4^2,4))=3$.
\end{proof}

\begin{lemma}
\label{4.19}
$fsmp(G(d^2,d))=fmp(G(d^2,d))=4$ for $d\geq 5$.
\end{lemma}

\begin{proof} If $d$ is even, then $|G(d^2,d)|=d^2$ is even and $d\geq 6$ as $d\geq 5$. By Lemma \ref{4.18},
$\alpha(G(d^2,d))\leq \frac{d^2-d}{2}\leq \frac{d^2+2}{2}-4$. Then $fsmp(T_{k_1,k_2})=fmp(T_{k_1,k_2})=4$ as Theorem \ref{fsmp}$(i)$. If $d$ is odd, then $|G(d^2,d)|=d^2$ is odd. By Lemma \ref{4.18}, $\alpha(G(d^2,d))\leq \frac{d^2-d}{2}\leq \frac{d^2+1}{2}-4+1$ for $d\geq 5$. Note that $G(d^2,d)\in \mathcal{G}_1(4)$, then $fsmp(T_{k_1,k_2})=fmp(T_{k_1,k_2})=4$ as Theorem \ref{fsmp}$(ii)$.
\end{proof}

\begin{lemma}
\label{4.20}
$fsmp(G(d^n,d))=fmp(G(d^n,d))=2n$ for $n\geq 3$.
\end{lemma}

\begin{proof} Note that $G(d^n, d)$ is $(2n)$-regular with $d^n$ vertices. By Lemmas \ref{4.18}, \ref{4.14} and Theorem \ref{fsmp}$(i)$, it suffices to show that
\begin{align}
\frac{d^n-d^{n-1}}{2}\leq \frac{d^n+1}{2}-2n,   \nonumber
\end{align}
which implies $d^{n-1}-4n+1\geq 0$. It is obvious that the inequality $d^{n-1}-4n+1\geq 0$ holds if $n\geq 4$. Particularly, if $n=3$, then the inequality $d^2-12+1\geq 0$ holds if $d\geq 4$. By Lemma \ref{4.18}, $\alpha(G(3^3,3))\leq 9=\frac{3^3+1}{2}-6+1$. Note that $G(3^3,3)\in \mathcal{G}_1(6)$, then $fsmp(G(3^3,3))=fmp(G(3^3,3))=6$ as Theorem \ref{fsmp}$(ii)$. Therefore, we complete the proof of Lemma \ref{4.20}.
\end{proof}

Combining with Lemmas~\ref{4.15},~\ref{4.16}, \ref{4.19} and \ref{4.20}, we obtain the FMP number and FSMP number of $G(d^n,d)$.

\begin{theorem}
\label{4.21}
Let $n\geq 2$ and $d\geq 3$ be integers. Then

\begin{eqnarray}
  fmp(G(d^n,d))=\left\{
\begin{array}{ll}
   3,        &\mbox{if}~(n,d)=(2,3);   \\          \nonumber
   2n,          &otherwise,                                                   \nonumber
   \end{array}\right.
\end{eqnarray}
and
\begin{eqnarray}
  fsmp(G(d^n,d))=\left\{
\begin{array}{ll}
   3,        &\mbox{if}~(n,d)\in \{(2,3),(2,4)\};   \\          \nonumber
   2n,          &otherwise.                                        \nonumber
   \end{array}\right.
\end{eqnarray}
\end{theorem}

\subsubsection{The recursive circulant graph $G(2^n, 4)$}

Park and Chwa \cite{Park94} introduced an interesting recursive circulant graph $G(2^n, 4)$, which is $n$-regular with $2^n$ vertices. By Lemma \ref{fmp}, $fmp(G(2^n, 4))=n$.

In the following, we establish the FSMP number of $G(2^n, 4)$. Park and Ihm \cite{Park11} showed that $G(2^n, 4)$ with odd $n$ is an $n$-dimensional restricted HL-graph, whose fractional (strong) matching preclusion properties were analyzed in \cite{Zhang19}.

\begin{lemma}\cite{Zhang19}
\label{4.22}
Let $n\geq 3$ be an odd. Then $fsmp(G(2^3, 4))=2$ and $fsmp(G(2^n, 4))=n$ for $n\geq 5$.
\end{lemma}

If $n$ is even, then by Theorem \ref{4.21}, $fsmp(G(2^4,4))=3$ and $fsmp(G(2^n,4))=n$ for $n\geq 6$. Combining this with Lemma \ref{4.22}, we obtain fractional (strong) matching preclusion of $G(2^n, 4)$ with $n\geq 3$ as follows.

\begin{theorem}
\label{4.23}
Let $n\geq 3$ be an integer. Then $fmp(G(2^n, 4))=n$ and
\begin{eqnarray}
  fsmp(G(2^n,4))=\left\{
\begin{array}{ll}
   n-1,           &\mbox{if}~n\in \{3,4\};      \\       \nonumber
   n,             &otherwise.                     \nonumber
   \end{array}\right.
\end{eqnarray}
\end{theorem}

%%%%%%%%%%%%%%%%%%%%%%%%%%%%%%%%%%%%%%%%%%%%%%%%%%%%%%%%%%%%%%%%%%%%%%%%%%%%%%%%%%%%%%%%%%%%%%%%%%%%%%%%%%%%%%%%%%%%%%%%%%%%%%%%%%%%%%%%%%%%%%%%%%%%%%%%%%%%%%%%%%%%%%%%%%%%%%%%%%%%

\subsection{$(n,k)$-arrangement graphs and $(n, k)$-star graphs}

The $(n,k)$-arrangement graph \cite{Day92} and $(n, k)$-star graph \cite{Chiang95} are two generalization versions of the star graph $S_n$. The two parameters $n$ and $k$ can be tuned to make a suitable choice for the number of nodes in the network and for the degree/diameter tradeoff.

\subsubsection{$(n,k)$-arrangement graphs}

\begin{definition} \cite{Day92}
\label{4.24}
The $(n,k)$-arrangement graph, denoted by $A_{n,k}$, is defined for positive integers $n$ and $k$ such that $1\leq k\leq n-1$. The vertex set of the graph
is all the permutations on $k$ elements of the set $\{1,2,\ldots,n\}$. Two vertices $a_1a_2\ldots a_k$
and $b_1b_2\ldots b_k$ are adjacent if there exists an integer $s$ with $1\leq s\leq k$ such
that $a_s\neq b_s$ and for any $i\neq s$, $1\leq i\leq k$, we have $a_i=b_i$.
\end{definition}

\begin{figure}[!htb]
\centering
{\includegraphics[height=0.45\textwidth]{7}}

\vskip .3cm

Fig.~7~~The $(4,2)$-arrangement graph $A_{4,2}$

\end{figure}

The $(4, 2)$-arrangement graph $A_{4,2}$ is shown in Fig.~7. $A_{n,1}$ is isomorphic to the complete graph $K_n$, $A_{n,n-2}$ are isomorphic to the $n$-alternating group graph $AG_n$ and $A_{n,n-1}$ is isomorphic to the $n$-dimensional star graph $S_n$. $A_{n,k}$ is $k(n-k)$-regular with $P^k_{n}=\frac{n!}{(n-k)!}$ vertices. In \cite{Day92}, Day and Tripathi proved that $A_{n,k}$ can be decomposed into $n$ vertex disjoint subgraphs, each of which is isomorphic to $A_{n-1,k-1}$ (see Fig.~7, each red triangle in $A_{4,2}$ is isomorphic to $A_{3,1}=K_3$).

\begin{lemma}
\label{4.26}
$\alpha(A_{n,k})\leq P^{k-1}_{n}$ for $1\leq k\leq n-2$.
\end{lemma}

\begin{proof} Clearly, $\alpha(A_{n,1})=1$. By induction on $k$, suppose that $\alpha(A_{n-1,k-1})\leq P^{k-2}_{n-1}$ for $2\leq k\leq n-2$. Note that $A_{n,k}$ can be decomposed into $n$ vertex disjoint subgraphs, each of which is isomorphic to $A_{n-1,k-1}$, and thus every independent set of $A_{n,k}$ contains at most $n\times P^{k-2}_{n-1}=P^{k-1}_{n}$ vertices. Then $\alpha(A_{n,k})\leq P^{k-1}_{n}$.
\end{proof}

Hsu {\em et al.} \cite{Hsu04} investigated fault Hamiltonicity of the arrangement graphs.

\begin{lemma}\cite{Hsu04}
\label{4.25}
$A_{n,k}$ is $(k(n-k)-2)$-fault Hamiltonian for $2\leq k\leq n-2$.
\end{lemma}

Now we can determine the FMP number and FSMP number of $A_{n,k}$, which was also obtained in \cite{MaMao19}.

\begin{theorem}
\label{4.27}
$fsmp(A_{n,k})=fmp(A_{n,k})=k(n-k)$ for $2\leq k\leq n-2$.
\end{theorem}

\begin{proof} Note that $A_{n,k}$ is $k(n-k)$-regular with $P^{k}_{n}$ vertices, and thus $|A_{n,k}|$ is even. By Lemmas \ref{4.26}, \ref{4.25} and Theorem \ref{fsmp}$(i)$, it suffices to show that
\begin{align}
P^{k-1}_{n}\leq \frac{P^{k}_{n}+2}{2}-k(n-k),   \nonumber
\end{align}
which implies $(n-k-1)P^{k-1}_{n}-2k(n-k)+2\geq 0$. It is obvious that the inequality $(n-k-1)P^{k-1}_{n}-2k(n-k)+2\geq 0$ holds if $A_{n,k}\neq A_{4,2}$. By Lemma \ref{4.26}, $\alpha(A_{4,2})\leq 4=\frac{12+2}{2}-4+1$. Note that $A_{4,2}\in \mathcal{G}_1(4)$, then $fsmp(A_{4,2})=fmp(A_{4,2})=4$ as Theorem \ref{fsmp}$(ii)$. Therefore, we complete the proof of Theorem \ref{4.27}.
\end{proof}

\subsubsection{$(n, k)$-star graphs}

\begin{definition} \cite{Chiang95}
\label{4.28}
The $(n,k)$-star graph, denoted by $S_{n,k}$, is defined for positive integers $n$ and $k$ such that $1\leq k\leq n-1$. The vertex set of the graph
is all the permutations on $k$ elements of the set $\{1,2,\ldots,n\}$. Two vertices corresponding to the permutations $a_1a_2\ldots a_k$
and $b_1b_2\ldots b_k$ are adjacent if and only if either:

$(1)$ There exists an integer $2\leq s\leq k$ such that $a_1=b_s$ and $b_1=a_s$ and for any $i\neq s$, $2\leq i\leq k$, we have $a_i=b_i$. That is,
$b_1b_2\ldots b_k$ is obtained from $a_1a_2\ldots a_k$ by swapping $a_1$ and $a_s$.

$(2)$ For all $2\leq i\leq k$, we have $a_i=b_i$ and $a_1\neq b_1$. That is, $b_1b_2\ldots b_k$ is obtained from $a_1a_2\ldots a_k$ by replacing $a_1$
by an element in $\{1,2,\ldots,n\}-\{a_1,a_2,\ldots,a_k\}$.
\end{definition}

\begin{figure}[!htb]
\centering
{\includegraphics[height=0.45\textwidth]{8}}

Fig.~8~~The $(4,2)$-star graph $S_{4,2}$

\end{figure}

The $(4,2)$-star graph $S_{4,2}$ is depicted in Fig.~8. $S_{n,1}$ is isomorphic to the complete graph $K_n$ and $S_{n,n-1}$ is isomorphic to the $n$-dimensional star graph $S_n$. $S_{n,k}$ is $(n-1)$-regular with $P^k_{n}=\frac{n!}{(n-k)!}$ vertices. In \cite{Chiang95}, Chiang and Chen proved that $S_{n,k}$ can be decomposed into $n$ vertex disjoint subgraphs, each of which is isomorphic to $S_{n-1,k-1}$ (see Fig.~8, each red triangle in $S_{4,2}$ is isomorphic to $S_{3,1}=K_3$).

\begin{lemma}
\label{4.30}
$\alpha(S_{n,k})\leq P^{k-1}_{n}$ for $1\leq k\leq n-2$.
\end{lemma}

\begin{proof} Clearly, $\alpha(S_{n,1})=1$. By induction on $k$, suppose that $\alpha(S_{n-1,k-1})\leq P^{k-2}_{n-1}$ for $2\leq k\leq n-2$. Note that $S_{n,k}$ can be decomposed into $n$ vertex disjoint subgraphs, each of which is isomorphic to $S_{n-1,k-1}$, and thus every independent set of $S_{n,k}$ contains at most $n\times P^{k-2}_{n-1}=P^{k-1}_{n}$ vertices. Then $\alpha(S_{n,k})\leq P^{k-1}_{n}$.
\end{proof}

Hsu {\em et al.} \cite{Hsu03} considered Hamiltonian properties of faulty $(n,k)$-star graphs.

\begin{lemma}\cite{Hsu03}
\label{4.29}
$S_{n,k}$ is $(n-3)$-fault Hamiltonian for $2\leq k\leq n-2$.
\end{lemma}

By Lemmas \ref{4.30}, \ref{4.29} and Theorem \ref{fsmp}$(i)$, we can determine the FMP number and FSMP number of $S_{n,k}$.

\begin{theorem}
\label{4.31}
$fsmp(S_{n,k})=fmp(S_{n,k})=n-1$ for $2\leq k\leq n-2$.
\end{theorem}

\subsection{(Burnt) pancake graphs}

Pancake graphs and burnt pancake graphs, introduced by Gates and Papadimitriou \cite{Gates79}, are two well-studied interconnection networks such as ring embedding,
super connectivity, broadcasting, fault-tolerant Hamiltonicity.

\subsubsection{Pancake graphs}

\begin{figure}[!htb]
\centering
{\includegraphics[height=0.45\textwidth]{9}}

Fig.~9~~The pancake graph $PG_4$

\end{figure}

\begin{definition} \cite{Gates79}
\label{4.32}
The pancake graph of dimension $n$, denoted by $PG_n$, has the set of all $n!$ permutations on $\{1,2,3,\ldots,n\}$ as its vertex set.
Two vertices $a_1a_2\ldots a_n$ and $b_1b_2\ldots b_n$ are adjacent if there exists an integer $k$ with $2\leq k\leq n$ such
that $a_i=b_{k+1-i}$ for every $i$ with $1\leq i\leq n$, and $a_i=b_i$ for $k+1\leq i\leq n$.
\end{definition}

Note that $PG_3$ is a $6$-cycle and $PG_4$ is given in Fig.~9. $PG_n$ is $(n-1)$-regular with $n!$ vertices and vertex transitive. Let $PG^i_{n-1}$ be the subgraph of $PG_n$ induced by the vertices with $i$ in the $n$-th position where $1\leq i\leq n$. Then $PG_n$ can be decomposed into $n$ vertex disjoint such subgraphs, each of which is isomorphic to $PG_{n-1}$ (see Fig.~9, each red cycle represents $PG^i_3$ for $1\leq i\leq 4$).

\begin{lemma}
\label{4.34}
$\alpha(PG_4)=10$.
\end{lemma}

\begin{proof} See Appendix.
\end{proof}

\begin{lemma}
\label{4.35}
$\alpha(PG_n)\leq \frac{5}{12}n!$ for $n\geq 4$.
\end{lemma}

\begin{proof} The conclusion holds for $n=4$ as Lemma \ref{4.34}. By induction, suppose that $\alpha(PG_{n-1})\leq \frac{5}{12}(n-1)!$ for $n\geq 5$. Note that $PG_n$ can be decomposed into $n$ vertex disjoint subgraphs each of which is isomorphic to $PG_{n-1}$, and thus every independent set of $PG_n$ contains at most $n\times \frac{5}{12}(n-1)!=\frac{5}{12}n!$ vertices. Then $\alpha(PG_n)\leq \frac{5}{12}n!$ for $n\geq 4$.
\end{proof}

Hung {\em et al.} \cite{Hung03} considered Hamiltonian properties in faulty pancake graphs.

\begin{lemma}\cite{Hung03}
\label{4.33}
$PG_n$ is $(n-3)$-fault Hamiltonian for $n\geq 4$.
\end{lemma}

By Lemmas \ref{4.35}, \ref{4.33} and Theorem \ref{fsmp}$(i)$, we can determine the FMP number and FSMP number of $PG_n$, which was also obtained in \cite{Ma18}.

\begin{theorem}
\label{4.36}
$fsmp(PG_n)=fmp(PG_n)=n-1$ for $n\geq 4$.
\end{theorem}

\subsubsection{Burnt pancake graphs}

We say the list $a_1a_2\ldots a_n$ is a signed permutation on $\{1, 2,3,\ldots,n\}$ if $|a_1||a_2|\ldots |a_n|$ is a permutation on $\{1,2,3,\ldots,n\}$. For notational simplicity, we use the notation $\overline{a}$ instead of $-a$ and $[n]$ instead of $\{1,2,3,\ldots,n\}\cup \{\overline{1},\overline{2},\overline{3},\ldots,\overline{n}\}$,

\begin{definition} \cite{Gates79}
\label{4.37}
The burnt pancake graph of $n$-dimension, denoted by $BP_n$, has the set of signed permutations on $\{1,2,3,\ldots,n\}$ as its vertex set.
Two vertices $a_1a_2\ldots a_n$ and $b_1b_2\ldots b_n$ are adjacent if there exists an integer $k$ with $1\leq k\leq n$ such
that $a_i=b_{\overline{k+1-i}}$ for every $i$ with $1\leq i\leq k$, and $a_i=b_i$ for $k+1\leq i\leq n$.
\end{definition}

\begin{figure}[!htb]
\centering
{\includegraphics[height=0.6\textwidth]{10}}

\vskip .3cm

Fig.~10~~the burnt pancake graph $BP_3$

\end{figure}

Note that $BP_2$ is an $8$-cycle and $BP_3$ is given in Fig.~10. $BP_n$ is $n$-regular with $n!2^n$ vertices and vertex transitive. Let $BP^i_{n-1}$ be the subgraph of $BP_n$ induced by the vertices with $i$ in the $n$-th position where $i\in [n]$. Then $BP_n$ can be decomposed into $2n$ vertex disjoint such subgraphs, each of which is isomorphic to $BP_{n-1}$ (see Fig.~10, each red even cycle represents $BP^i_2$ for $i\in [3]$).

\begin{lemma}
\label{4.39}
$\alpha(BP_3)=20$.
\end{lemma}

\begin{proof} See Appendix.
\end{proof}

\begin{lemma}
\label{4.40}
$\alpha(BP_n)\leq \frac{5}{12}n!2^n$ for $n\geq 3$.
\end{lemma}

\begin{proof} The conclusion holds for $n=3$ as Lemma \ref{4.39}. By induction, suppose that $\alpha(BP_{n-1})\leq \frac{5}{12}(n-1)!2^{n-1}$ for $n\geq 4$. Note that $BP_n$ can be decomposed into $2n$ vertex disjoint subgraphs, each of which is isomorphic to $BP_{n-1}$, and thus every independent set of $BP_n$ contains at most $2n\times \frac{5}{12}(n-1)!2^{n-1}=\frac{5}{12}n!2^n$ vertices. Then $\alpha(BP_n)\leq \frac{5}{12}n!2^n$ for $n\geq 3$.
\end{proof}

Kaneko \cite{Kaneko07} considered Hamiltonian properties in faulty burnt pancake graphs.

\begin{lemma}\cite{Kaneko07}
\label{4.38}
$BP_n$ is $(n-2)$-fault Hamiltonian for $n\geq 3$.
\end{lemma}

By Lemmas \ref{4.40}, \ref{4.38} and Theorem \ref{fsmp}$(i)$, we can determine the FMP number and FSMP number of $BP_n$, which was also obtained in \cite{Ma18}.

\begin{theorem}
\label{4.41}
$fsmp(BP_n)=fmp(BP_n)=n$ for $n\geq 3$.
\end{theorem}

%%%%%%%%%%%%%%%%%%%%%%%%%%%%%%%%%%%%%%%%%%%%%%%%%%%%%%%%%%%%%%%%%%%%%%%%%%%%%%%%%%%%%%%%%%%%%%%%%%%%%%%%%%%%%%%%%%%%%%%%%%%%%%%%%%%%%%%%%%%%%%%%%%%%%%%%%%%%%%%%%%%%%%%%%%%%%%%%%%%%

\section{Conclusions}

In this paper, we establish the FMP number and FSMP number of fault Hamiltonian graphs. Let $G$ be a $(\delta-2)$-fault Hamiltonian graph with minimum degree $\delta\geq 3$. If $\alpha(G)\leq \lceil \frac{|G|+1}{2}\rceil -\delta$, then $fsmp(G)=fmp(G)=\delta$; If $G\in \mathcal{G}_1(\delta)$ and $\alpha(G)\leq \lceil \frac{|G|+1}{2}\rceil -\delta+1$, then $fsmp(G)=fmp(G)=\delta$; If $G\in \mathcal{G}_2(4)$ and $\alpha(G)\leq \frac{|G|+1}{2}-2$, then $fsmp(G)=fmp(G)=4$. As applications, the FMP number and FSMP number of some well-known networks, such as the restricted HL-graph $G^n$, the $n$-dimensional torus $T_{k_1,\cdots,k_n}$, the recursive circulant graphs $G(d^n,d)$ and $G(2^n,4)$, the $(n,k)$-arrangement graph $A_{n,k}$, the $(n, k)$-star graph $S_{n,k}$, the pancake graph $PG_n$ and the burnt pancake graph $BP_n$, are determined (see Table~1).

\begin{table}[!hbp]

\vskip.1cm

\begin{tabular}{|c|c|c|}
\hline
Networks                                                                                   &the FMP number                 & the FSMP number \\
\hline       $G^n$($n\ge 5$)                                                               &  $n$                          & $n$ \\
\hline       Bipartite\,$T_{k_1,\cdots,k_n}$($n\ge 2$,\,$k_i\ge 3$)                        &  $2n$                         & 1 \\
\hline       \multirow{2}{*}{Non-bipartite\,$T_{k_1,\cdots,k_n}$($n\ge 2$,\,$k_i\ge 3$)}   &  \multirow{2}{*}{$2n$ }       & ~~$3$,~if\,$(n,k_2)=(2,4)$  \\
\cline{3-3}  \multirow{2}{*}{}                                                             &  \multirow{2}{*}{ }           & $2n$,~if\,$(n,k_2)\neq(2,4)$  \\
\hline       \multirow{2}{*}{$G(d^n,d)$($n\ge 2$,\,$d\ge 3$)}                            & ~~$3$,~if\,$(n,d)=(2,3)$      & ~~$3$,~if\,$(n,d)\in\{(2,3),(2,4)\}$\\
\cline{2-3}  \multirow{2}{*}{ }                                                            & $2n$,~if\,$(n,d)\neq(2,3)$    & $2n$,~if\,$(n,d)\notin\{(2,3),(2,4)\}$\\
\hline       \multirow{2}{*}{$G(2^n,4)$($n\ge 3$)  }                                       & \multirow{2}{*}{$n$}          & $n-1$,~if\,$n\in \{3,4\}$\\
\cline{3-3}  \multirow{2}{*}{ }                                                            & \multirow{2}{*}{ }            & ~~~~~$n$,~if\,$n\notin\{3,4\}$\\
\hline       $A_{n,k}$($2\le k\le n-2$)                                                    & $k(n-k)$                      & $k(n-k)$\\
\hline       $S_{n,k}$($2\le k\le n-2$)                                                    & $n-1$                         & $n-1$\\
\hline       $PG_n$($n\ge 4$)                                                              & $n-1$                         & $n-1$\\
\hline       $BP_n$($n\ge 3$)                                                              & $n$                           & $n$\\
\hline
\end{tabular}

\caption{The FMP number and FSMP number of some well-known networks}
\end{table}

%%%%%%%%%%%%%%%%%%%%%%%%%%%%%%%%%%%%%%%%%%%%%%%%%%%%%%%%%%%%%%%%%%%%%%%%%%%%%%%%%%%%%%%%%%%%%%%%%%%%%%%%%%%%%%%%%%%%%%%%%%%%%%%%%%%%%%%%%%%%%%%%%%%%%%%%%%%%%%%%%%%%%%%%%%%%%%%%%%%%

%%%%%%%%%%%%%%%%%%%%%%%%%%%%%%%%%%%%%%%%%%%%%%%%%%%%%%%%%%%%%%%%%%%%%%%%%%%%%%%%%%%%%%%%%%%%%%%%%%%%%%%%%%%%%%%%%%%%%%%%%%%%%%%%%%%%%%%%%%%%%%%%%%%%%%%%%%%%%%%%%%%%%%%%%%%%%%%%%%%%

\noindent{\bf\Large Appendix}

\vskip.4cm

\noindent {\bf Proof of Lemma 4.14.} First we show that $\alpha(G(d^2,d))\geq \frac{d^2-d}{2}$. Denote $I=I_1\cup I_2$ if $d$ is odd, where $I_1=
\{v_{2i,2j}: 0\leq i\leq \frac{d-3}{2}~\mbox{and}~0\leq j\leq \frac{d-1}{2}\}$ and $I_2=\{v_{2i-1,2j-1}: 1\leq i, j\leq \frac{d-1}{2}\}$. Otherwise, $I=I'_1\cup I'_2$ if $d$ is even, where $I'_1=\{v_{2i,2j}: 0\leq i,j\leq \frac{d-2}{2}\}$ and $I'_2=\{v_{2i-1,2j-1}: 1\leq i\leq \frac{d}{2}~\mbox{and}~1\leq j\leq \frac{d-2}{2}\}$. It follows that $I$ is an independent set, and thus $\alpha(G(d^2,d))\geq |I|=\frac{d^2-d}{2}$.

\begin{figure}[!htb]
\centering
{\includegraphics[height=0.46\textwidth]{11}}

\vskip .3cm

Fig.~11~~$G(d^2,d)$ can be decomposed into $d-1$ vertex disjoint $(d+1)$-cycles\\ and a vertex $v_{d-1,d-1}$.

\end{figure}

Next we prove $\alpha(G(d^2,d))\leq \frac{d^2-d}{2}$. Recall that $G^1_{d,i}$ is isomorphic to $G(d, d)$ which is a cycle with $d$ vertices for $0\leq i<d$. If $d$ is odd, then $G(d^2,d)$ has a spanning subgraph consisting of $d$ vertex disjoint odd cycles $G^1_{d,0},\ldots,G^1_{d,d-1}$. Note that every independent set of the odd cycle $G^1_{d,i}$ contains at most $\frac{d-1}{2}$ vertices for $0\leq i< d$, then every independent set of $G(d^2,d)$ contains at most $d\times \frac{d-1}{2}=\frac{d^2-d}{2}$ vertices. Thus $\alpha(G(d^2,d))\leq \frac{d^2-d}{2}$.

If $d$ is even, then we denote $C_k=v_{k-1,k-1}\ldots v_{k-1,d-1}v_{k,0}\ldots v_{k,k-1}v_{k-1,k-1}$ for $1\leq k<d$, and thus $C_k$ is a ($d+1$)-cycle. It follows that $G(d^2,d)$ has a spanning subgraph consisting of $d-1$ vertex disjoint odd cycles $C_1,\ldots,C_{d-1}$ and a vertex $v_{d-1,d-1}$(see Fig.~11, each red cycle represents odd cycle $C_k$ for $1\leq k< d$). Since every independent set of the odd cycle $C_k$ for $1\leq k<d$ contains at most $\frac{d}{2}$ vertices, every independent set of $G(d^2,d)$ contains at most $(d-1)\times \frac{d}{2}+1=\frac{d^2-d}{2}+1$ vertices. Thus $\alpha(G(d^2,d))\leq \frac{d^2-d}{2}+1$. Now, suppose that $\alpha(G(d^2,d))=\frac{d^2-d}{2}+1$. Then there exists an independent set $I'$ of $G(d^2,d)$ such that $|I'|=\frac{d^2-d}{2}+1$. This implies that $v_{d-1,d-1}\in I'$ and $|V(C_k)\cap I'|=\frac{d}{2}$ for $1\leq k<d$. Denote $P_1:=C_1-\{v_{0,0},v_{0,d-1},v_{1,0}\}$. Since $P_1$ is an even path with order $d-2$, we have $|V(P_1)\cap I'|\leq \frac{d}{2}-1$. This together with $|V(C_1)\cap I'|=\frac{d}{2}$, we can deduce that $|\{v_{0,0},v_{0,d-1},v_{1,0}\}\cap I'|\geq 1$. Note that $v_{0,0}v_{d-1,d-1},v_{0,d-1}v_{d-1,d-1}\in E(G(d^2,d))$ and $v_{d-1,d-1}\in I'$. This implies $v_{1,0}\in I'$. By the similar argument above, we have $v_{l,l-1}\in I'$ for $1\leq l\leq d-2$ and $v_{d-2,d-2}\in I'$ (see Fig.~11, the green dots represent the vertices in $I'$). But $v_{d-2,d-3}v_{d-2,d-2}\in E(G(d^2,d))$, a contradiction. Then $\alpha(G(d^2,d))\neq \frac{d^2-d}{2}+1$, and thus $\alpha(G(d^2,d))\leq \frac{d^2-d}{2}$.  \q

\vskip .5cm

\noindent {\bf Proof of Lemma 4.33.} First we show that $\alpha(PG_4)\geq 10$. Denote $I=\{1234,3124,3412,\\1342,4132,2413,4123,4231,2341,3421\}$ (see Fig.~9, the green dots represent the vertices of $I$). It follows that $I$ is an independent set of $PG_4$, and thus $\alpha(PG_4)\geq |I|=10$.

\begin{figure}[!htb]
\centering
{\includegraphics[height=0.46\textwidth]{12}}

\vskip .3cm

Fig.~12~~$PG_4$ can be decomposed into four vertex disjoint $6$-cycles.

\end{figure}

Next we prove $\alpha(PG_4)\leq 10$. Recall that $PG_4$ can be decomposed into four vertex disjoint $PG^1_3,PG^2_3,PG^3_3,PG^4_3$, each of which is isomorphic to a $6$-cycle (see Fig.~12, each red cycle represents $PG^i_3$ for $1\leq i\leq 4$). Thus $\alpha(PG^i_3)=3$ for $1\leq i\leq 4$. Let $I'$ be a maximum independent set of $PG_4$. Suppose that $\alpha(PG_4)\geq 11$. Then there exists at least three elements of $\{1,2,3,4\}$ such that $|I'\cap V(PG^i_3)|=3$ for $1\leq i\leq 4$. By symmetry, assume that $|I'\cap V(PG^i_3)|=3$ for $1\leq i\leq 3$.
Then $I'\cap V(PG^1_3)=\{4231,2341,3421\}$ or $\{3241,2431,4321\}$. Without loss of generality, assume that $I'\cap V(PG^1_3)=\{4231,2341,3421\}$. Note that $2341$ and $3421$ are adjacent to $1432$ and $1243$, respectively. Thus $I'\cap V(PG^2_3)=\{3412,1342,4132\}$ and $I'\cap V(PG^3_3)=\{1423,4213,2143\}$ (see Fig.~12, the green dots represent the vertices in $I'$). But $3412$ is adjacent to $2143$ in $PG_4$, a contradiction. Then $\alpha(PG_4)\leq 10$.

\begin{figure}[!htb]
\centering
{\includegraphics[height=0.6\textwidth]{13}}

\vskip .3cm

Fig.~13~~$BP_3$ can be decomposed into six vertex disjoint $8$-cycles.

\end{figure}

\vskip .5cm

\noindent {\bf Proof of Lemma 4.38.}  First we show that $\alpha(BP_3)\geq 20$. Denote $I=\{2\overline{13},\overline21\overline3,312,\overline132,\\1\overline32,\overline{31}2,\overline{231},3\overline{21},23\overline1,\overline32\overline1,123,213,\overline{132},\overline31\overline2,13\overline2,3\overline{12},321,2\overline31,\overline{32}1,\overline231\}$ (see Fig.~10, the\\ green dots represent the vertices of $I$). It follows that $I$ is an independent set of $BP_3$, and thus $\alpha(BP_3)\geq |I|=20$.

Next we prove $\alpha(BP_3)\leq 20$. Recall that $BP_3$ can be decomposed into six vertex disjoint $BP^i_2$ with $i\in [3]$, each of which is isomorphic to an $8$-cycle (see Fig.~13, each red cycle represents $PG^i_3$ for $i\in [3]$). Thus $\alpha(BP^i_2)=4$ for $i\in [3]$. Let $I'$ be a maximum independent set of $BP_3$. Suppose that $\alpha(BP_3)\geq 21$. Then there exists at least three elements of $[3]$ such that $|I'\cap V(BP^i_2)|=4$ for $i\in [3]$. Without loss of generality, assume that $|I'\cap V(BP^1_2)|=4$ and $I'\cap V(BP^{1}_2)=\{321,2\overline31,\overline{32}1,\overline231\}$. Now, we consider the following two cases.

\vskip 0.2cm
{\bf Case 1.} $|I'\cap V(BP^i_2)|=|I'\cap V(BP^{\overline{i}}_2)|=4$ for some $i\in \{1,2,3\}$.
\vskip 0.2cm

By symmetry, assume that $|I'\cap V(BP^1_2)|=|I'\cap V(BP^{\overline{1}}_2)|=4$. Recall that there exists at least three elements of $[3]$ such that $|I'\cap V(BP^i_2)|=4$ for $i\in [3]$. Without loss of generality, assume that $|I'\cap V(BP^{\overline{2}}_2)|=4$. Since $2\overline31$ is adjacent to $\overline13\overline2$, $I'\cap V(BP^{\overline2}_2)=\{\overline31\overline2,13\overline2,3\overline{12},\overline{132}\}$. Note that $13\overline2$ is adjacent to $2\overline{31}$ and $|I'\cap V(BP^{\overline{1}}_2)|=4$, then $I'\cap V(BP^{\overline1}_2)=\{\overline{231},3\overline{21},23\overline1,\overline32\overline1\}$. This implies $\overline{123},21\overline3,12\overline3\notin I'\cap V(BP^{\overline3}_2)$, $\overline{13}2, 132\notin I'\cap V(BP^{2}_2)$ and $1\overline23,2\overline13,\overline123\notin I'\cap V(BP^{3}_2)$ (see Fig.~13, the green dots represent the vertices in $I'$ and the black crosses represent the vertices out of $I'$). Thus $|I'\cap V(BP^{\overline3}_2)|\leq 3$ and $|I'\cap V(BP^{3}_2)|\leq 3$. Since $\alpha(BP_3)=|I'|\geq 21$, we have $3\leq |I'\cap V(BP^{2}_2)|\leq 4$.

Suppose that $|I'\cap V(BP^{2}_2)|=4$. Then $I'\cap V(BP^{2}_2)=\{312,\overline132,\overline{31}2,1\overline32\}$. This implies $\overline{213}\notin I'\cap V(BP^{\overline3}_2)$ and $123\notin I'\cap V(BP^{3}_2)$. Hence, $I'\cap V(BP^{\overline3}_2)\leq 2$ and $|I'\cap V(BP^{3}_2)|\leq 2$. Then $|I'|\leq 4\times 4+2\times 2=20$, a contradiction. Thus $|I'\cap V(BP^{2}_2)|=3$. Recall that $|I'\cap V(BP^{\overline3}_2)|\leq 3$ and $|I'\cap V(BP^{3}_2)|\leq 3$. Then $|I'\cap V(BP^{\overline3}_2)|=|I'\cap V(BP^{3}_2)|=3$ as $\alpha(BP_3)\geq 21$. It follows that $\overline{213}\in I'\cap V(BP^{\overline3}_2)$ and $\overline213\in I'\cap V(BP^{3}_2)$, and thus $312,\overline{31}2\notin I'\cap V(BP^{2}_2)$. But $|I'\cap V(BP^{2}_2)|\leq 2$, a contradiction.

\vskip 0.2cm
{\bf Case 2.} $|I'\cap V(BP^i_2)|\leq 3$ or $|I'\cap V(BP^{\overline{i}}_2)|\leq 3$ for any $i\in \{1,2,3\}$.
\vskip 0.2cm

Recall that there exists at least three elements of $[3]$ such that $|I'\cap V(BP^i_2)|=4$ for $i\in [3]$. Suppose that $|I'\cap V(BP^{\overline2}_2)|=4$ or $|I'\cap V(BP^{\overline3}_2)|=4$. Without loss of generality, assume that $|I'\cap V(BP^{\overline2}_2)|=4$. Since $2\overline31$ is adjacent to $\overline13\overline2$, $I'\cap V(BP^{\overline2}_2)=\{\overline31\overline2,13\overline2,3\overline{12},\overline{132}\}$. Note that $321$ and $3\overline{12}$ are adjacent to $\overline{123}$ and $21\overline{3}$. Thus $|I'\cap V(BP^{\overline3}_2)|\leq 3$. This implies $|I'\cap V(BP^{3}_2)|=4$. Since $\overline31\overline2$ is adjacent to $2\overline13$, $I'\cap V(BP^{3}_2)=\{1\overline23,\overline{21}3,\overline123,213\}$. But $\overline{32}1$ is adjacent to $\overline{1}23$ in $BP_3$, a contradiction. Then $|I'\cap V(BP^{\overline2}_2)|\leq 3$ and $|I'\cap V(BP^{\overline3}_2)|\leq 3$. Hence, $|I'\cap V(BP^{2}_2)|=|I'\cap V(BP^{3}_2)|=4$. Since $\overline231$ is adjacent to $\overline{13}2$, $I'\cap V(BP^{2}_2)=\{312,\overline132,\overline{31}2,1\overline32\}$. Since $\overline{32}1$ is adjacent to $\overline123$, $I'\cap V(BP^{3}_2)=\{2\overline13,123,\overline213,\overline{12}3\}$. But $\overline213$ is adjacent to $\overline{31}2$ in $BP_3$, a contradiction.

\vskip 0.4cm

%%%%%%%%%%%%%%%%%%%%%%%%%%%%%%%%%%%%%%%%%%%%%%%%%%%%%%%%%%%%%%%%%%%%%%%%%%%%%%%%%%%%%%%%%%%%%%%%%%%%%%%%%%%%%%%%%%%%%%%%%%%%%%%%%%%%%%%%%%%%%%%%%%%%%%%%%%%%%%%%%%%%%%%%%%%%%%%%%%%%

\end{document}